\documentclass[%
sn-mathphys-num,pdflatex]{sn-jnl}%

\usepackage{graphicx}%
\usepackage{multirow}%
\usepackage[tbtags]{amsmath}%
\usepackage{amssymb}%
\usepackage{amsfonts}%
\usepackage{mathrsfs}%
\usepackage[title]{appendix}%
\usepackage{xcolor}%
\usepackage{textcomp}%
\usepackage{manyfoot}%
\usepackage{booktabs}%
\usepackage{algorithm}%
\usepackage{algorithmicx}%
\usepackage{algpseudocode}%
\usepackage{listings}%

\usepackage{dsfont} %
\newcommand{\dd}{\;\mathrm{d}}
\allowdisplaybreaks[4]
\usepackage[pagewise]{lineno}

\theoremstyle{thmstyleone}%
\newtheorem{theorem}{Theorem}[section]%
\newtheorem{proposition}[theorem]{Proposition}%
\newtheorem{lemma}[theorem]{Lemma}

\theoremstyle{thmstyletwo}%
\newtheorem*{remark}{Remark}%

\theoremstyle{thmstylethree}%

\begin{document}

\title[Finite-time blowup in indirect chemotaxis]{Finite-time blowup in a parabolic-parabolic-elliptic chemotaxis model involving indirect signal production}

\author*[1]{\fnm{Xuan} \sur{Mao}}\email{20250611@hhu.edu.cn}

\author[2]{\fnm{Yuxiang} \sur{Li}}\email{lieyx@seu.edu.cn}
\equalcont{These authors contributed equally to this work.}

\affil*[1]{\orgdiv{School of Mathematics}, \orgname{Hohai University}, \orgaddress{ \city{Nanjing} \postcode{211100}, \country{China}}}

\affil[2]{\orgdiv{School of Mathematics}, \orgname{Southeast University}, \orgaddress{ \city{Nanjing} \postcode{211189}, \country{China}}}

\abstract{This paper is concerned with a three-component chemotaxis model accounting for indirect signal production, 
reading as $u_t = \nabla\cdot(\nabla u - u\nabla v)$, 
$v_t = \Delta v - v + w$ and $0 = \Delta w - w + u$, 
posed in a ball of $\mathbb R^n$ with $n\geq5$, 
subject to homogeneous Neumann boundary conditions.
The system is a Nagai-type variant of its fully parabolic version that has a four-dimensional critical mass phenomenon concerning blowup in finite or infinite time 
according to the seminal works of Fujie and Senba [J. Differential Equations, 263
(2017), 88--148; 266 (2019), 942--976]. 
We prove that for any prescribed mass $m > 0$, there exist radially symmetric and positive initial data $(u_0,v_0)\in C^0(\overline{\Omega})\times C^2(\overline{\Omega})$ with $\int_\Omega u_0 = m$ such that the corresponding solutions blow up in finite time. }



\maketitle

\section{Introduction}\label{introduce section}

This paper is concerned with finite-time singularity formation of classical solutions to the following parabolic-parabolic-elliptic chemotaxis model accounting for indirect signal production 
\begin{align}
  \begin{cases}
    \label{sys: ks isp pp/ep/e}
      u_t = \Delta u - \nabla \cdot(u\nabla v),&  x\in\Omega, t>0,\\
      v_t =  \Delta v - v + w,&   x\in\Omega,	t>0,\\
      0  = \Delta w - w + u, &   x\in\Omega, t > 0,\\
      \partial_\nu u = \partial_\nu v = \partial_\nu w = 0 , &  x\in\partial\Omega, t >0,\\
      (u(\cdot, 0), v(\cdot, 0)) = (u_0, v_0), & x\in\Omega,
  \end{cases}
\end{align}
in a bounded domain $\Omega\subset\mathbb R^n$ with smooth boundary for some $n\in\mathbb N$, 
where $\partial_\nu$ denotes the derivative with respect to the outward normal of $\partial\Omega$,
and the initial functions $u_0\in C^0(\overline\Omega)$ and 
$v_0\in C^2(\overline\Omega)$ are nonnegative.

Among the chemotaxis models that describe the oriented movement of organisms in response to chemical stimuli, 
the minimal chemotaxis system%
\begin{equation}
  \label{sys: classical ks}
	\begin{cases}
		u_t = \nabla\cdot(\nabla u - u\nabla v), & x\in\Omega, t > 0, \\
		v_t = \Delta v - v + u,                  & x\in\Omega, t > 0,
	\end{cases}
\end{equation}
has been extensively studied.
This system was originally proposed by Keller and Segel~\cite{Keller1970} to investigate the aggregation of cells.
The density $u$ of cells is influenced by their motion, which consists of random walk and chemotactic movement along the gradient of chemical concentration $v$.
As is well-known, the system~\eqref{sys: classical ks} exhibits a competitive mechanism between diffusion and aggregation induced by chemotaxis where blowup---the extreme form of aggregation---can occur in finite time under certain conditions: 
\begin{itemize} 
\item In one dimension, diffusion dominates aggregation, and all solutions of \eqref{sys: classical ks} are globally bounded~\cite{Osaki2001}. 
\item In a planar domain, blowup may occur in finite or infinite time if the initial mass $m := \int_\Omega u_0$ exceeds the critical mass $4\pi$ but not equals to an integer multiple of $4\pi$ \cite{Horstmann2001}; all solutions remain globally bounded if $m < 4\pi$ (or $8\pi$ with symmetry assumptions) \cite{Nagai1997}. 
For a disk-shaped domain $\Omega$, there exist radially symmetric initial data with mass $m\in(8\pi,\infty)$ such that the solutions blow up in finite time \cite{Mizoguchi2014}.
\item In higher dimensions, diffusion is rather governed by chemotaxis, and for any initial mass $m > 0$, there exist initial data with mass $m$ that lead to finite-time blowup under radial settings~\cite{Winkler2013}.
\end{itemize}

Recently, chemotaxis models involving indirect signal production,
often expressed in the following typical form
\begin{align}
  \begin{cases}
    \label{sys: ks isp ppp}
      u_t =  \Delta u - \nabla \cdot(u\nabla v) + f(u,w),& x\in\Omega, t>0,\\
      v_t =  \Delta v - v + w,& x\in\Omega,	 t>0,\\
      w_t  = \tau\Delta w - w + u, & x\in\Omega, t > 0, 
  \end{cases}
  \end{align}
have received significant attention
for their ability to characterize more complex and realistic biological phenomena. 
For instance, in a study of predicting the cluster attack behaviors of mountain pine beetles \cite{Strohm2013}, 
the beetle population was divided into two distinct phenotypes: one represented by $u$ for beetles performing chemotaxis and the other by $w$ for beetles producing the attractant $v$,
with $f$ modeling potential growth, death and phenotype switching~\cite{Macfarlane2022}. 
Here,
$\tau\in\{0,1\}$ indicates the motility of the signal producers. 

When the signal producer is static, i.e., $\tau = 0$, 
the indirect attractive mechanism has an explosion-suppressing (or boundedness-supporting) effect. 
Hu and Tao \cite{Hu2016} showed that 
quadratic logistic dampening $f(u,w) = \mu u(1-u)$ to any weak extent $\mu > 0$ can prevent blowup in a three-dimensional indirect chemotaxis model with growth \eqref{sys: ks isp ppp}.  
In contrast, Winkler~\cite{Winkler2010} established boundedness in the three- and higher-dimensional parabolic-parabolic chemotaxis system with a strong logistic source in the sense of $\mu > \mu_\star$ for some $\mu_\star > 0$.
For the mass-preserving system \eqref{sys: ks isp ppp} with $f\equiv0$, a novel two-dimensional critical mass phenomenon has been identified, 
concerning infinite-time blowup as documented by Tao and Winkler~\cite{Tao2017} and Lauren\c{c}ot~\cite{Laurencot2019}.
For chemotaxis models \eqref{sys: ks isp ppp} with phenotype switching $f = -u+w$, the occurrence of two-dimensional infinite-time blowup depends on the total mass $\int_\Omega u_0 + \int_\Omega v_0$ \cite{Laurencot2021}.

For fully parabolic indirect-chemotaxis models ($\tau=1$) without a source term $f(u,w)\equiv0$, 
the indirect attractant mechanism can also reduce explosion-supporting potential.
Fujie and Senba~\cite{Fujie2017,Fujie2019} established global boundedness in physical spaces $n\leq3$ and identified a four-dimensional critical mass phenomenon for classical solutions of \eqref{sys: ks isp ppp}. 
Specifically, they found that  
\begin{itemize}
  \item  
  Under symmetry assumptions,
  the solutions exists globally and remain bounded 
  for nonnegative initial data $(u_0,v_0,w_0)$ with $\int_\Omega u_0 < 64\pi^2$. 
  Without radial symmetry, the same conclusion holds for \eqref{sys: ks isp ppp} under mixed boundary conditions, where  
  $\partial_\nu u - u\partial_\nu v = v = w = 0$ on $\partial\Omega$ for $t>0$.
  \item For any $m \in(64\pi^2,\infty)\setminus 64\pi^2\mathbb N$, there exist initial data satisfying 
  $\int_\Omega u_0 = m$ such that 
  the solutions of \eqref{sys: ks isp ppp} under mixed boundary conditions blow up in finite or infinite time.
\end{itemize}
Beyond four dimensions, Mao and Li~\cite{Mao2025} showed that for any $m>0$, there exist initial data with $\int_\Omega u_0 = m$ such that the solutions of the system \eqref{sys: ks isp ppp} blow up in finite or infinite time with symmetry assumptions.
We emphasize both critical dimension $4$ and critical mass $64\pi^2$ of \eqref{sys: ks isp ppp} regarding blowup are significantly larger compared to those of the classical chemotaxis model \eqref{sys: classical ks}. This increase is due to the strong parabolic smoothing effect of indirect attractant production.
Chen and Xiang~\cite{Chen2024} observed a similar positive effect on boundedness in certain indirect chemotaxis-haptotaxis model. Additionally, taxis-driven aggregation, particularly the spontaneous emergence of large densities during the intermediate transition triggered by large phenotype switching rates, has been explored in \cite{Painter2023,Laurencot2024}.

Concerning finite-time singularity formation, Tao and Winkler~\cite{Tao2025} considered a J\"ager-Luckhaus variant %
of the system \eqref{sys: ks isp ppp}, where the second equation was modified to $\Delta v - \int_\Omega w/|\Omega| + w = 0$, and the source term was taken as $f = -au + bw$ for some constants $a\geq0$ and $b\geq0$.
They showed that there exist radially symmetric initial data such that 
solutions of the variant, posed in a ball of five and higher dimensions, blow up in finite time. This was deduced by robust comparison methods, which are applicable to cooperative systems involving volume-filling %
and phenotype switching. %
In this regard, we focus on detecting finite-time blowup for \eqref{sys: ks isp pp/ep/e}, a Nagai-type variant %
of \eqref{sys: ks isp ppp}. 
Here, we adopt the assumption that the change of $w$ is stationary over time \cite{Jin2016}.
Notably, the system~\eqref{sys: ks isp ppp} is also explained as the balanced cases between attraction and repulsion in competing chemotaxis that models the aggregation of microglia observed in Alzheimer's disease~~\cite{Tao2013,Fujie2017,Fujie2019,Ding2019}.%

\subsection*{Main ideas and results}

The system~\eqref{sys: ks isp pp/ep/e} has an energy dissipation structure identified by Fujie and Senba~\cite{Fujie2017}.
The Lyapunov functional
\begin{equation*}
  \mathcal{F}(u, v) := \int_{\Omega}(u \log u- u v)
+\frac{1}{2} \int_{\Omega}|(-\Delta+1) v|^2, 
\end{equation*}
satisfies the differential identity  
\begin{equation*} 
\frac{\dd}{\dd t} \mathcal{F}(u(t), v(t))+\mathcal{D}(u(t), v(t))=0
\end{equation*}
along the reasonably regular solution curves of the system \eqref{sys: ks isp pp/ep/e}.
Here, the dissipation rate reads as 
\begin{equation*}
\mathcal{D}(u, v) :=
\int_{\Omega}\left(\left|\nabla (\Delta v - v + w)\right|^2+\left|\Delta v - v + w\right|^2\right)
+\int_{\Omega} u|\nabla(\log u - v)|^2,
\end{equation*}
where $w := (-\Delta_N + 1)^{-1}u$ is the unique classical solution of the Neumann elliptic problem
\begin{equation}
  \label{sys: w elliptic}
  \begin{cases}
    - \Delta w + w = u, & x\in\Omega,\\
    \partial_\nu w = 0, &x\in\partial\Omega.
  \end{cases}
\end{equation}
Invoking the above mathematical structure,
we are inspired by previous works (see Remark~\ref{r: precedents} below) to derive a lower bound for $\mathcal{F}$ in terms of a sublinear power of $\mathcal{D}$ in the form of 
\begin{equation}
  \label{eq: mathcalFCDtheta}
  -\mathcal{F} \leq C(\mathcal{D}^\theta +1) 
  \quad\text{with some }\theta\in(0,1)\text{ and } C > 0,
\end{equation}
for a large class of radially symmetric function pairs $(u,v)=(u(r),v(r))$, which, in particular, includes the radial trajectories of the system~\eqref{sys: ks isp pp/ep/e}.  
This result leads to a superlinear differential inequality for $-\mathcal{F}$,
which, in turn, implies occurrence of finite-time blowup for initial data with low energy. 

To be precise, we shall suppose the initial data $(u_0, v_0)$ consist of nonnegative functions satisfying 
\begin{equation}
  \label{h: initial data}
  0 \lneqq u_0\in C^0(\overline{\Omega}),
  \quad 0\leq v_0\in C^2(\overline{\Omega})
  \quad \text{with } \partial_\nu v_0 \equiv 0 \text{ on } \partial\Omega
\end{equation}
and additionally, we impose symmetry assumptions that 
\begin{equation}
  \label{h: radial initial data}
  (u_0,v_0) \text{ is a pair of positive and radially symmetric functions on } \overline{\Omega}, 
\end{equation}
with $\Omega = B_R := \{x\in\mathbb{R}^n\mid |x| < R\}$ for some $R>0$ and $n\in\mathbb{N}$.

Our main result states that low initial energy enforces finite-time blowup.

\begin{theorem}
  \label{thm: low energy enforced finite-time blowup}
  Assume $\Omega = B_R \subset\mathbb R^n$ for some $n\geq5$ and $R>0$.
  Let $m > 0$ and $A > 0$ be given.
  Then there exist positive constants $K(m,A)$ and $T(m,A)$ such that 
  given any $(u_0,v_0)$ from the set 
  \begin{equation}
    \label{sym: mathcal B}
    \begin{aligned}
       &\mathcal{B}(m,A) := \{(u_0,v_0)\in  C^0(\overline{\Omega}) \times C^2(\overline{\Omega})\mid 
        u_0 \text{ and } v_0 \text{ comply with } \eqref{h: initial data} 
       \text{ and } 
       \eqref{h: radial initial data} \\
        &\text{ such that } \|u_0\|_{L^1(\Omega)} = m, 
        \|v_0\|_{W^{2,2}(\Omega)} < A
        \text{ and } \mathcal{F}(u_0,v_0) < - K(m,A)\},
    \end{aligned}
  \end{equation}
  for the corresponding classical solution with maximal existence time $T_{\max}(u_0,v_0)$ given by Proposition~\ref{prop: local existence and uniqueness},
  one has $T_{\max}(u_0,v_0) \leq T(m,A)$.
\end{theorem}

Concerning the existence of initial data with low energy,
for any given positive initial datum with suitable regularity,
we construct a sequence of initial data with arbitrarily large negative energy,
which can approximate the given initial datum in an appropriate topology. 

\begin{theorem}
  \label{thm: dense initial data for blowup}
  Assume $\Omega = B_R \subset\mathbb R^n$ for some $n\geq5$ and $R>0$.
  If $u_0\in C^0(\overline{\Omega})$ and $v_0\in C^2(\overline{\Omega})$ comply with \eqref{h: initial data} and \eqref{h: radial initial data},
  then
  one can find a sequence
  \begin{displaymath}
    \{(u_k, v_k) \in L^\infty(\Omega) \times W^{2,\infty}(\Omega)\}_{k\in\mathbb N}
  \end{displaymath}
  of function couples with properties:
  $(u_k, v_k)$ satisfying \eqref{h: initial data}, \eqref{h: radial initial data} and $\int_\Omega u_k = \int_\Omega u_0$ for all $k\in\mathbb N$ is such that 
  \begin{displaymath}
    \|u_k - u_0\|_{L^p(\Omega)}
    + \|v_k - v_0\|_{W^{2,2}(\Omega)} 
    \to 0
    \quad\text{as } k\to\infty, 
  \end{displaymath}
  holds for all $p\in\Big[1,\frac{2n}{n+4}\Big)$, 
  but $\mathcal F(u_k, v_k)\to-\infty$ as $k\to\infty$.
\end{theorem}

\begin{remark}
  \label{r: precedents}
  Our methods are built on the original ideas of Winkler \cite{Winkler2013},
  and are also inspired by previous works where functional inequalities of the type in \eqref{eq: mathcalFCDtheta} have been successfully applied to chemotaxis models in various contexts: volume-filling by Cie\'{s}lak and Stinner~\cite{Cieslak2012,Cieslak2014,Cieslak2015} and by Cao and Fuest~\cite{Cao2025}; the Cauchy problem by Winkler~\cite{Winkler2020a};
  competing chemotaxis by Lankeit~\cite{Lankeit2021}; two species models by Li and Li~\cite{Li2014}. 
  As far as we are aware, this work presents the first systematic implementation of the functional inequality \eqref{eq: mathcalFCDtheta} within the context of indirect chemotaxis.

  The primary challenge lies in deriving the functional inequality \eqref{eq: mathcalFCDtheta} due to explosion-suppressing effect of indirect attractant mechanisms. 
  Mathematically, we shall obtain interior estimates $\|\Delta v\|_{L^2(B_\rho)}$ in terms of the dissipation rate $\mathcal{D}$ for $\rho\in(0,R)$. To achieve this, we first establish weighted uniform estimates $\||x|^{\beta-1} v\|_{W^{1,\infty}(\Omega)}$ with some $\beta > n-2$ in Section~\ref{sec: pointwise}.
  In Section~\ref{sec: linking energy to dissipation}, 
  we are stimulated by \cite{Winkler2013,Cabre2020} to introduce the Poho\v{z}aev-type test function $\mathds{1}_{B_\rho}|x|^{n-2}x\cdot\nabla v$ to obtain the more complex second-order estimate. Here and below, $\mathds{1}_E$ stands for the indicator function of a Borel set $E\subseteq\mathbb R^n$.
\end{remark}

This paper is organized as follows. In Section~\ref{section preliminary},
we present some preliminaries involving the energy functional and basic estimates.
Section~\ref{sec: pointwise} is dedicated to deriving uniform-in-time pointwise estimates for the chemical concentration. 
In Section~\ref{sec: linking energy to dissipation}, we establish the functional inequality \eqref{eq: mathcalFCDtheta}.
Section~\ref{sec: finite-time blowup} is devoted to the proof of Theorem~\ref{thm: low energy enforced finite-time blowup}.
In Section~\ref{sec: initial data}, 
we construct a sequence of initial data with arbitrarily large negative energy 
thus completing the proof of Theorem~\ref{thm: dense initial data for blowup}.

\section{Preliminaries}
\label{section preliminary}

We begin by revisiting the results on the local existence and uniqueness of classical solutions to the system \eqref{sys: ks isp pp/ep/e}.

\begin{proposition}
  \label{prop: local existence and uniqueness}
  Let $n\in\mathbb N$ 
  and $\Omega \subset \mathbb{R}^n$ be a bounded domain with smooth boundary. 
 Assume that $(u_0, v_0)$ is as in \eqref{h: initial data}.
 Then there exist $T_{\max } \in(0, \infty]$ and a unique triplet $(u, v, w)$ of nonnegative functions from 
 $C^0\left(\bar{\Omega} \times\left[0, T_{\max }\right)\right) \cap C^{2,1}\left(\bar{\Omega} \times\left(0, T_{\max }\right)\right)$ 
 solving \eqref{sys: ks isp pp/ep/e} classically in $\Omega \times\left(0, T_{\max }\right)$. 
 Moreover, $u>0$, $v>0$ and $w>0$ in $\Omega \times\left(0, T_{\max }\right)$, and 
 \begin{equation} 
 \text{if } T_{\max } < \infty,
  \text{ then }
  \|u(\cdot, t)\|_{L^{\infty}(\Omega)}\to\infty 
  \text{ as } t \nearrow T_{\max }.
 \end{equation}
If in addition $\Omega = B_R$ for some $R>0$,
and $(u_0, v_0)$ is a pair of radially symmetric functions, 
then $u$, $v$ and $w$ are all radially symmetric.
\end{proposition}

\begin{proof}
  The local existence, uniqueness and extensibility criterion of classical solutions to the system~\eqref{sys: ks isp pp/ep/e} 
  can be established by adapting well-known fixed-point arguments in the contexts of chemotaxis models, as outlined in \cite[Theorem~3.1]{Horstmann2005}, \cite[Lemma~3.1]{Tao2013} and \cite[Section~4]{Fujie2017}. 
  The nonnegativity follows from the maximum principle, 
  while conservation of radial symmetry is a consequence of the uniqueness of solutions and the adequate form of the equations in \eqref{sys: ks isp ppp}.
\end{proof}

A Lyapunov functional associated with the system \eqref{sys: ks isp pp/ep/e} was formulated in~\cite[Proposition~6.1]{Fujie2017}.

\begin{proposition}
  Let $(u,v,w)$ be a classical solution given as in Proposition~\ref{prop: local existence and uniqueness}. 
  Then the following identity holds:
\begin{equation} 
  \label{eq: energyequation}
\frac{\dd}{\dd t} \mathcal{F}(u(t), v(t))+\mathcal{D}(u(t), v(t))=0 \quad \text { for all } t \in(0, T_{\max})
\end{equation}
where
\begin{equation}
  \label{sym: mathcalFD}
\begin{split} 
& \mathcal{F}(u, v)=\int_{\Omega}(u \log u- u v)
+\frac{1}{2} \int_{\Omega}|(-\Delta+1) v|^2\quad\text{and} \\
& \mathcal{D}(u, v)=
\int_{\Omega}\left(\left|\nabla v_t\right|^2+\left|v_t\right|^2\right)
+\int_{\Omega} u|\nabla(\log u - v)|^2 .
\end{split}
\end{equation}
\end{proposition}

For the sake of simplicity in notation, 
we shall collect some time-independent estimates specifically for high dimensions, i.e., $n\geq4$. These estimates will serve as a starting point for obtaining certain pointwise estimates in Section~\ref{sec: pointwise}. 

\begin{lemma}
  \label{le: basic estimates}
  Let $\Omega = B_R$ for some $R>0$ and $n\geq4$. 
  Then it holds that for any classical solution $(u,v,w)$ given as in Proposition~\ref{prop: local existence and uniqueness},
  \begin{equation}
    \label{eq: mass identity}
    \int_{\Omega} u(x, t) =\int_{\Omega} u_{0} = \int_\Omega w(x,t) =: m
      \quad\text{for all } t \in (0, T_{\max})
  \end{equation}
  and 
  \begin{equation}
    \label{eq: vmassinequality}
    \int_\Omega v(x,t) 
    \leq \max\left\{\int_\Omega v_0, \int_\Omega u_0\right\}
    \quad \text{for all } t\in(0,T_{\max}).
  \end{equation}
Furthermore, for each $p\in\Big[1,\frac{n}{n-2}\Big)$
and $q\in\Big[1,\frac{n}{n-1}\Big)$, 
there exist constants $C(p)>0$ and $C(q)>0$ such that 
for any choice of initial data $(u_0,v_0)$ with~\eqref{h: initial data} and \eqref{h: radial initial data},
the solution of \eqref{sys: ks isp pp/ep/e} satisfies 
  \begin{equation}
    \label{eq: wW1q}
    \|w\|_{W^{1,q}(\Omega)} \leq C(q)m
    \quad\text{for all } t \in (0, T_{\max}),
  \end{equation}
  and 
  \begin{equation}
    \label{eq: vW2p}
    \|v_t\|_{L^p(\Omega)}
    + \|w\|_{L^p(\Omega)}
    + \|v\|_{W^{2,p}(\Omega)} 
    \leq C(p)(m + \|v_0\|_{W^{2,2}(\Omega)})
    \quad\text{for all } t \in (0, T_{\max}).
  \end{equation}
\end{lemma}

\begin{proof}
The mass identity \eqref{eq: mass identity} follows directly 
from integrating the first and last equations in \eqref{sys: ks isp ppp} over $\Omega$.
Integrating the second equation yields  
\begin{equation*}
  \int_\Omega v(x,t) 
  = e^{-t}\int_\Omega v_0  
  + (1-e^{-t})\int_\Omega u_0 
  \leq \max\left\{\int_\Omega v_0, \int_\Omega u_0\right\}
  \quad \text{for all } t\in(0,T_{\max}),
\end{equation*}
which shows \eqref{eq: vmassinequality}.
By applying well-known $W^{1,p}$ regularity theory \cite{Brezis1973} to 
the elliptic boundary value problem $-\Delta w + w = u$, 
we obtain that for each 
$q\in\Big[1,\frac{n}{n-1}\Big)$,
there exists a constant $C(q)>0$ such that $\|w\|_{W^{1,q}(\Omega)} \leq C(q)\|u\|_{L^1(\Omega)}$
holds for all $t\in(0,T_{\max})$.
This implies \eqref{eq: wW1q} by virtue of \eqref{eq: mass identity}.
Thus, $w\in L^p(\Omega)$ for $p\in\Big[1,\frac{n}{n-2}\Big)$ by Sobolev embedding $W^{1,q}\hookrightarrow L^{\frac{nq}{n-q}}$ for $q\in\Big[1,\frac{n}{n-1}\Big)$. 
Moreover, for any $p\in\Big[1,\frac{n}{n-2}\Big)$ one can select 
$q\in\Big[1,\frac{n}{n-1}\Big)$ such that 
\begin{equation}
  \label{eq: pqchoices}
  -\frac{1}{2}-\frac{n}{2}\left(\frac{1}{q}-\frac{1}{p}\right) > -1.
\end{equation}
$L^p$--$L^q$ estimates of the heat semigroup $e^{t\Delta}$ 
with homogeneous Dirichlet boundary conditions (cf. \cite[Lemma~2.1]{Wu2022} and \cite[Lemma~2.4]{Mu2024}) 
applied to $v_{x_it} = \Delta v_{x_i} - v_{x_i} + w_{x_i}$ 
subject to $v_{x_i} = v_rx_i/r \equiv 0$ on $\partial\Omega$ for $i \in\{1,2,\ldots,n\}$,
provide a constant $C_d>0$ such that 
  \begin{align*}
    \|\nabla v_{x_i}\|_{L^p(\Omega)} 
    &\leq \|\nabla e^{t(\Delta-1)}v_{0x_i}\|_{L^{p}(\Omega)} 
    + \int_0^t\|\nabla e^{(t-\tau)(\Delta-1)}w_{x_i}\|_{L^p(\Omega)}\dd \tau\\
    &\leq |\Omega|^{1-p/2}\|\nabla e^{t(\Delta-1)}v_{0x_i}\|_{L^2(\Omega)} \\
    &\quad + C_d\sup_{\tau\in(0,t)}\|w_{x_i}(\cdot,\tau)\|_{L^q(\Omega)} \cdot \int_0^t\left(1+(t-\tau)^{-\frac{1}{2}-\frac{n}{2}\left(\frac{1}{q}-\frac{1}{p}\right)}\right)e^{-(t-\tau)}\dd \tau \\
    &\leq C_d\|v_0\|_{W^{2,2}(\Omega)} 
    + C_dC(q)m 
    \bigg(1+\int_0^\infty \tau^{-\frac{1}{2}-\frac{n}{2}\left(\frac{1}{q}-\frac{1}{p}\right)}e^{-\tau}\dd \tau\bigg),
  \end{align*}
holds for all $t\in(0,T_{\max})$, $p\in\Big[1,\frac{n}{n-2}\Big)$ and $i \in\{1,2,\ldots,n\}$,
where the last integral is finite due to the choices of parameters in~\eqref{eq: pqchoices}. 
According to \eqref{eq: vmassinequality},
this implies that $v\in W^{2,p}(\Omega)$ for all $t\in(0,T_{\max})$, 
and \eqref{eq: vW2p} holds,
since $v_t = \Delta v - v + w \in L^p(\Omega)$ for all $t\in(0,T_{\max})$.
\end{proof}

\section{Pointwise estimates for chemical concentration}\label{sec: pointwise}

This section is devoted to the deduction of uniform-in-time pointwise estimates for $v$ and $w$,
based on Lemma~\ref{le: basic estimates}.
Specifically, we show that symmetric functions from Sobolev spaces can be embedded into certain weighted Sobolev spaces, as established in~\cite{GuedesdeFigueiredo2011}. 
For convenience, we shall switch to the radial notation without any further comment, e.g., writing $v (r, t)$ instead of $v (x, t)$ when appropriate.
Throughout the sequel, we assume that the radius $R$ of the ball $\Omega$ is fixed. 
Consequently, it will not be explicitly stated that generic constants $C(p), C(m,M,B,\kappa), \ldots$
may vary from line to line,
and additionally depend on $R$.

The main results of this section are presented as follows.

\begin{proposition}
\label{prop: pointwise estimate}
  Let $\Omega = B_R\subset\mathbb R^n$ with some $R > 0$ and $n\geq4$. 
  There exists a constant $C_w > 0$ such that for any radially symmetric initial datum $(u_0,v_0)$ with \eqref{h: initial data} and \eqref{h: radial initial data}, 
  the corresponding classical solution $(u,v,w)$ of \eqref{sys: ks isp pp/ep/e} satisfies 
  \begin{equation}
    \label{eq: wpointwise}
    \sup_{t\in(0,T_{\max})} \||x|^{n-1}w\|_{W^{1,\infty}(\Omega)} \leq C_wm.
  \end{equation}
  For each $\beta > n - 2$, there exists $C(\beta) > 0$ such that 
  \begin{equation}
    \label{eq: vpointwise}
    \sup_{t\in(0,T_{\max})} \||x|^{\beta-1} v\|_{W^{1,\infty}(\Omega)}
    \leq C(\beta)(m + \|v_0\|_{W^{2,2}(\Omega)})
  \end{equation} 
  holds whenever radially symmetric initial data $(u_0,v_0)$ comply with \eqref{h: initial data} and \eqref{h: radial initial data}.
\end{proposition}

\begin{proof}
  The weighted uniform estimate \eqref{eq: wpointwise} follows directly from Lemma~\ref{le: wpointwise}.
  The estimates in Lemma~\ref{le: gradvpointwise} yield \eqref{eq: vpointwise}.
\end{proof}

The mass identity \eqref{eq: mass identity} permits us to gain pointwise information about $w$.   

\begin{lemma}
  \label{le: wpointwise}
  Let $\Omega = B_R\subset\mathbb R^n$ with some $R > 0$ and $n\geq4$. 
  Then there exists a constant $C_w > 0$ such that 
  \begin{equation}
    w + |x\cdot\nabla w| \leq \frac{C_wm}{|x|^{n-2}},\quad |x|\in(0,R],\quad t\in(0,T_{\max})
  \end{equation}
  for all $w$ given in Proposition~{\upshape\ref{prop: pointwise estimate}}.
\end{lemma}

\begin{proof}
  Integrating the last equation in \eqref{sys: ks isp pp/ep/e}, $ur^{n-1} = wr^{n-1} - (w_rr^{n-1})_r$ over the interval $(0,r)\subset(0,R)$ yields 
  \begin{displaymath}
    |w_r| = \left|r^{1-n}\left(\int_0^rw\eta^{n-1}\dd\eta 
    - \int_0^ru\eta^{n-1}\dd\eta\right) \right|
    \leq \frac{2\omega_n^{-1}m}{ r^{n-1}}
  \end{displaymath}
  for all $t\in(0,T_{\max})$,
  where the mass identity \eqref{eq: mass identity} has been used.
  Here and henceforth, $\omega_n$ denotes the surface area of the unit ball in $\mathbb R^n$.
  Choose $r_0(t)\in (R/2,R)$ such that  
  \begin{displaymath}
    w(r_0) \leq \frac{1}{|B_R\setminus B_{R/2}|}\int_{B_R\setminus B_{R/2}}w
    \quad\text{for all } t\in(0,T_{\max}).
  \end{displaymath}
  Using the Newton-Leibniz formula,
  we obtain  
  \begin{displaymath}
    \begin{split}
      w = w(r_0) + \int_{r_0}^rw_r(\eta)\dd\eta 
      &\leq \frac{m}{|B_R\setminus B_{R/2}|} 
      + \frac{2\omega_n^{-1}m}{n-2}\left(r^{2-n} + r_0^{2-n}\right)\\
      &\leq \frac{mR^{n-2}r^{2-n}}{|B_R\setminus B_{R/2}|} 
      + \frac{2\omega_n^{-1}m}{n-2}\left(r^{2-n} + 2^{n-2}r^{2-n} \right)
    \end{split}
  \end{displaymath}
  for all $t\in(0,T_{\max})$.
  Setting 
  \begin{equation*}
    C_w := \frac{R^{n-2}}{|B_R\setminus B_{R/2}|} + \frac{2^n}{\omega_n},
  \end{equation*}
  we complete the proof.
\end{proof}

Following the standard approach, e.g. in \cite{Ni1982,Winkler2013}, 
we can deduce pointwise estimates for the gradient of $v$ from \eqref{eq: vW2p}.
See \cite[Theorem~1.1]{GuedesdeFigueiredo2011} for generalized results concerning Sobolev imbedding in the spaces of radial functions.

\begin{lemma}
  \label{le: gradvpointwise}
For any $\beta > n - 2$, 
there exists $C(\beta) > 0$ such that 
  \begin{displaymath}
    \frac{v}{|x|^2} + 
  \frac{|\nabla v|}{|x|} 
  \leq \frac{C(\beta)}{|x|^{\beta}}\left(m + \|v_0\|_{W^{2,2}(\Omega)}\right), 
  \quad |x|\in(0,R], \quad t\in(0,T_{\max})
  \end{displaymath}
  holds for all $v$ given in Proposition~{\upshape\ref{prop: pointwise estimate}}. 
\end{lemma}

\begin{proof}
  Since $\beta > n-2$, we can put 
\begin{equation}
  \label{sym: p}
  p := \max\bigg\{\frac{n}{n-1},\frac{n}{\beta}\bigg\}\in\bigg(1, \frac{n}{n-2}\bigg).
\end{equation}
H\"older inequality entails 
\begin{align*}
  &\left|\int_{r_1}^r v_{rr}(s)+ \frac{n-1}{s}v_r(s)\dd s\right|\\
  &\leq \left(\int_{0}^R |v_{rr}(s)+ \frac{n-1}{s}v_r(s)|^p s^{n-1}\dd s\right)^{1/p}
      \left| \int_{r_1}^r s^{-(n-1)/(p-1)} \dd s\right|^{\frac{p-1}p}\\
  &\leq \omega_n^{-1/p}\|\Delta v\|_{L^p(\Omega)}
  \left(r^{-\frac{n-p}{p}} + r_1^{-\frac{n-p}{p}}\right)
\end{align*}
and similarly,
\begin{align*}
  \left|\int_{r_1}^r \frac{n-1}{s}v_r(s)\dd s\right|
  &\leq (n-1)\left(\int_{0}^R|v_r|^{\frac{np}{n-p}}s^{n-1}\dd s\right)^{\frac{n-p}{np}} 
  \left|\int_{r_1}^rs^{-\frac{n^2-n+p}{np-n+p}}\dd s\right|^{\frac{np-n+p}{np}}\\
  &\leq (n-1)\omega_n^{-\frac{n-p}{np}}\|\nabla v\|_{L^{\frac{np}{n-p}}(\Omega)}
  \left(r^{-\frac{n-p}{p}} + r_1^{-\frac{n-p}{p}}\right),
\end{align*}
for all $r\in(0,R)$ and $t\in(0,T_{\max})$, 
where 
$r_1(t)\in(R/2,R)$ is chosen such that 
\begin{displaymath}
  |v_r(r_1)| \leq \frac{1}{|B_R\setminus B_{R/2}|}\int_{B_R\setminus B_{R/2}}|\nabla v|,
  \quad t\in(0,T_{\max}).
\end{displaymath}
Newton-Leibniz formula gives 
\begin{align*}
  |v_r| &= \left|v_r(r_1) + \int_{r_1}^r v_{rr}(s)\dd s\right|\\
  &\leq  \left|v_r(r_0)\right| 
  + \left|\int_{r_1}^r v_{rr}(s)+ \frac{n-1}{s}v_r(s)\dd s\right| 
  + \left|\int_{r_1}^r \frac{n-1}{s}v_r(s)\dd s\right| \\ 
  &\leq 
  C_s\|v\|_{W^{2,p}(\Omega)}\left(r^{-\frac{n-p}{p}} + r_1^{-\frac{n-p}{p}} + \frac{1}{|B_R\setminus B_{R/2}|}\right),
\end{align*}
for all $r\in(0,R)$ and $t\in(0,T_{\max})$, 
with some $C_s>0$ inferred from Sobolev imbedding $W^{2,p}(\Omega)\hookrightarrow W^{1,\frac{np}{n-p}}(\Omega) \hookrightarrow W^{1,1}(\Omega)$.
According to \eqref{eq: vW2p} and \eqref{sym: p}, the desired pointwise estimates follow from an analogous argument used in Lemma~\ref{le: wpointwise}.
\end{proof}

\section{Energy bounded below by sublinear dissipation rate}
\label{sec: linking energy to dissipation}

The goal of this section is to establish a lower bound for the energy functional $\mathcal{F}$ in terms of a power of the dissipation rate $\mathcal D$, built upon the functional framework originally proposed by Winkler~\cite{Winkler2013},
as discussed in the introduction,
where $\mathcal{F}$ and $\mathcal{D}$ are defined in \eqref{sym: mathcalFD}.

To relate the energy $\mathcal{F}$ to the dissipation rate $\mathcal{D}$ under symmetry assumptions, 
we introduce the following quantities 
\begin{equation}
  \label{sym: f}
  f := - \Delta v + v - w 
\end{equation}
and 
\begin{equation}
  \label{sym: g}
  g := \left(\frac{\nabla u}{\sqrt u} - \sqrt u \nabla v\right)\cdot \frac{x}{|x|},
  \quad x\in\overline{\Omega}\setminus\{0\},
\end{equation} 
where 
\begin{equation}
  \label{sym: w}
  w := (-\Delta_N + 1)^{-1}u
\end{equation}
is the unique classical solution to the Neumann elliptic problem~\eqref{sys: w elliptic}.
Then the energy functional $\mathcal{F}$ and the dissipation rate $\mathcal{D}$ are given by 
\begin{equation} 
  \label{sym: mathacl F D}
\begin{split} 
& \mathcal{F}(u, v)=\int_{\Omega}(u \log u- u v)
+\frac{1}{2} \int_{\Omega}|(-\Delta+1) v|^2 \quad \text{and} \\
& \mathcal{D}(u, v)=
\int_{\Omega}\left(\left|\nabla f\right|^2+\left|f\right|^2\right)
+\int_{\Omega} g^2 .
\end{split}
\end{equation}

According to Lemma~\ref{le: basic estimates} and Proposition~\ref{prop: pointwise estimate},
we will focus on positive and radially symmetric functions $(u,v)$ that satisfy the mass constraints 
\begin{equation}
  \label{eq: constraint of mass}
  \int_\Omega u = m, 
  \quad \int_\Omega v \leq M 
  \quad\text{and } L^1\text{-bound}\quad 
  \int_\Omega |f| \leq B,
\end{equation}
as well as the additional pointwise restrictions 
\begin{equation}
  \label{eq: pointwise constraint for v}
  \||x|^{n-1}w\|_{W^{1,\infty}(\Omega)} \leq M 
  \quad\text{and}\quad 
  \||x|^{\kappa-1} v\|_{W^{1,\infty}(\Omega)}
  \leq B,
\end{equation}
where $m>0$, $M>0$, $B>0$ and $\kappa > n - 2$ are fixed parameters.
More precisely, the objective is to derive an inequality of the form
\begin{displaymath} 
\frac{\mathcal{F}(u, v)}{\mathcal{D}^\theta(u, v)+1} \geqslant-C(m, M, B, \kappa) \quad \text { for all }(u, v) \in \mathcal{S}(m, M, B, \kappa)
\end{displaymath}
where $\theta \in(0,1)$ and $C(m, M, B, \kappa)>0$. 
Here and below, $\mathcal{S}(m, M, B, \kappa)$ is the set of function pairs defined as 
\begin{equation*}
\begin{aligned}
\mathcal{S}(m, M, B, \kappa)
&:= \{(u, v) \in C^1(\bar{\Omega}) \times C^3(\bar{\Omega}) \mid 
 u \text { and } v \text { are positive and radially } \\
&\:\text { symmetric functions satisfying \eqref{eq: constraint of mass}, \eqref{eq: pointwise constraint for v} and } \partial_\nu v \equiv 0 \text{ on } \partial\Omega\}.
\end{aligned}
\end{equation*}

We begin with the following straightforward observation: 
controlling the energy functional $\mathcal F$ from below 
is equivalent to bounding the mixed term $\int_\Omega uv$ from above.

\begin{lemma}
  \label{le: firstobservation}
  Let $\Omega = B_R\subset\mathbb R^n$ be a ball of radius $R>0$ with some $n\geq4$.
  Then 
  \begin{equation}
    -\mathcal{F}(u,v) \leq \int_\Omega uv 
    + \frac{\omega_nR^n}{e}
  \end{equation}
  holds for all $(u,v)\in \mathcal S(m,M,B,\kappa)$, 
  where $\omega_n$ is the surface area of the unit ball in $\mathbb{R}^n$.
\end{lemma}

\begin{proof}
  It follows from the basic inequality $s\ln s \geq - e^{-1}$ valid for all $s>0$.
\end{proof}

We incorporate decoupling techniques from Cao and Fuest~\cite{Cao2025},
which assist us to convert the mixed term $\int_\Omega uv$ into an interior $L^2$--estimate for the Laplacian of the chemical concentration $v$. We will handle quadratic zeroth-order terms later.

\begin{lemma}
  \label{le: uvL1smallball}
  Let $\Omega = B_R\subset\mathbb R^n$ be a ball of radius $R>0$ with some $n\geq4$.
  There exists a constant $C_2 > 0$ such that 
  \begin{equation*}
    \int_{\Omega} uv \leq 3 \|\Delta v\|_{L^2(B_\rho)}^2   
      + 3\|v\|_{L^2(B_\rho)}^2
      + \|f\|_{L^2(B_\rho)}^2
      + C_2(m+M)B\rho^{2-\kappa},
  \end{equation*}
  holds for all $(u,v)\in\mathcal S(m,M,B,\kappa)$ and all $\rho\in(0,R)$.
\end{lemma}

\begin{proof}
In view of \eqref{sym: w}, we test $u = -\Delta w + w$ by $v\mathds{1}_{B_\rho}$, 
and compute upon integration by parts,  
  \begin{equation}
    \label{eq: uvpsi}
    \begin{split}
      \int_{B_\rho} uv 
      &= -\int_{B_\rho} \Delta w v 
      + \int_{B_\rho} wv
      = - \int_{\partial B_\rho} v\partial_\nu w\dd S 
      + \int_{B_\rho} \nabla w\cdot \nabla v 
      + \int_{B_\rho} wv\\ 
      &= \int_{B_\rho} (-\Delta v + v)w 
      + \int_{\partial B_\rho} (w\partial_\nu v - v\partial_\nu w)\dd S =: J_1 + J_2.
    \end{split}
  \end{equation}
  Using $w = -\Delta v + v - f$ specified in \eqref{sym: f},
  we estimate by Young inequality
  \begin{equation}
    \begin{split}
      J_1 &= \int_{B_\rho} (-\Delta v + v)^2 
      + \int_{B_\rho} (\Delta v - v)f
      \leq \frac32\|\Delta v - v\|_{L^2(B_\rho)}^2 
       + \frac12\|f\|_{L^2(B_\rho)}^2\\ 
      &\leq 3\|\Delta v\|_{L^2(B_\rho)}^2 
      + 3\|v\|_{L^2(B_\rho)}^2 
      + \|f\|_{L^2(B_\rho)}^2.
    \end{split}
  \end{equation}
  Recalling pointwise information from \eqref{eq: pointwise constraint for v}, 
  we estimate 
  \begin{equation}
      J_2 \leq MB \int_{\partial B_\rho}
      (|x|^{2-n}|x|^{1-\kappa} + |x|^{2-\kappa}|x|^{1-n})\dd S
      = 2\omega_nMB\rho^{2-\kappa},
  \end{equation}
 and also estimate the integral in the annular region
  \begin{equation}
    \label{eq: uvannulus}
    \int_{\Omega \backslash B_{\rho}} uv 
    \leq \sup _{r \in\left(\rho, R\right)} v(r) \cdot \int_{\Omega \setminus B_{\rho}} u 
    \leq B \sup_{r \in\left(\rho, R\right)} r^{2-\kappa} \cdot \int_{\Omega} u 
    = Bm \rho^{2-\kappa}.
    \end{equation}
  Summing \eqref{eq: uvpsi}--\eqref{eq: uvannulus} yields the desired estimate 
  \begin{displaymath}
    \int_{\Omega} uv \leq 3 \|\Delta v\|_{L^2(B_\rho)}^2   
    + 3\|v\|_{L^2(B_\rho)}^2
    + \|f\|_{L^2(B_\rho)}^2
    + (2\omega_nMB + Bm)\rho^{2-\kappa},
  \end{displaymath}
  for all $(u,v)\in \mathcal{S}(m,M,B,\kappa)$ and $\rho\in(0,R)$,
  which completes the proof.
\end{proof}

To leverage the structural information \eqref{sym: g}, 
we employ the Poho\v{z}aev-type test function $|x|^{n-2}x\cdot\nabla v$ over the inner region $B_\rho$,
initially introduced in the seminal work of Winkler~\cite{Winkler2010b, Winkler2010,Winkler2013}.
This strategy enables us to establish an upper bound for the interior estimate $\|\Delta v\|_{L^2(B_\rho)}$ through a combination of dissipation rate quantities and lower-order signal components. 
Specifically, this procedure generates a small factor $\rho^2$ in the critical term $\rho^2\|\nabla f\|_{L^2(B_\rho)}^2$.
Through careful calibration (see Lemma~\ref{le: mathcalFDtheta}), we shall demonstrate that 
$\rho$ can be adaptively selected depending on $\|\nabla f\|_{L^2(\Omega)}$ 
to effectively convert this quadratic term into a sublinear power of the dissipation rate quantity $\|\nabla f\|_{L^2(\Omega)}^2$.  
Such test function methodologies have become standard tools in modern elliptic analysis, as evidenced by recent developments in \cite{Cabre2020}.

\begin{lemma}
  \label{le: deltavL2}
  Let $\Omega = B_R\subset\mathbb R^n$ be a ball of radius $R>0$ with some  $n\geq5$. 
  Then there exists a constant $C_0 > 0$ such that
  \begin{equation} 
    \label{eq: 1/8deltavL2}
  \begin{split}
    &\frac{1}{8} \|\Delta v\|_{L^2(B_\rho)}^2  
    + \frac{3}{4}  \|\nabla v\|_{L^2(B_\rho)}^2 
    \leq 12\rho^2\|\nabla f\|_{L^2(B_\rho)}^2 
    + C_0\|f\|_{L^2(B_\rho)}^2\\
    &+ \sqrt m\rho\|g\|_{L^2(B_\rho)} 
     + \|v\|_{L^2(B_\rho)}^2
    + C_0B^2\rho^{n+2-2\kappa} + C_0BM\rho^{2-\kappa} + 2m 
  \end{split} 
  \end{equation}
  holds for all $(u,v)\in \mathcal{S}(m,M,B,\kappa)$ and $\rho\in(0,R)$.
\end{lemma}

\begin{proof}
  Denote $\alpha = 2(n-1)$.
  Multiplying $\Delta w - w + u = 0$ by $|x|^{n-2}x\cdot\nabla v$ 
  and integrating over $B_\eta\subset\Omega$ for $\eta\in(0, R)$,
  we get 
  \begin{align}
    \label{eq: test by rn-1vr}
    - \int_0^\eta uv_rr^\alpha\dd r + \int_0^\eta wv_rr^\alpha\dd r
    &=  \int_0^\eta (w_rr^{n-1})_rv_rr^{n-1}\dd r.
  \end{align}
  Integration by parts yields 
  \begin{equation} 
  \begin{split} 
    \int_0^\eta (w_rr^{n-1})_rv_rr^{n-1}\dd r
    &= w_r(\eta)v_r(\eta)\eta^\alpha 
    - \int_0^\eta w_rr^{n-1}(v_rr^{n-1})_r\dd r \\
    &= w_r(\eta)v_r(\eta)\eta^\alpha 
    - \int_0^\eta w_ryr^\alpha\dd r, \quad \eta\in(0, R).
  \end{split} 
\end{equation}
  Here, we abbreviate   
  \begin{displaymath}
    y := \Delta v = v_{rr} + \frac{n-1}{r}v_r = (v_rr^{n-1})_rr^{1-n},\quad r\in(0, R)
  \end{displaymath}
  for brevity.
  Using $w = - y + v -f$ and $w_r = - y_r + v_r - f_r$,
  we evaluate for $\eta\in(0, R)$
  \begin{equation} 
    \label{eq: wvrr2n-2}
  \begin{split}
    &\int_0^\eta wv_rr^\alpha\dd r 
    = \int_0^\eta (-y + v -f)v_rr^\alpha\dd r \\
    &= - \int_0^\eta (v_rr^{n-1})_rv_rr^{n-1}\dd r 
    + \frac{1}{2}\int_0^\eta(v^2)_rr^\alpha\dd r
    - \int_0^\eta fv_rr^\alpha\dd r \\
    &= - \frac12v_r^2 (\eta)\eta^\alpha 
    + \frac12 v^2(\eta)\eta^\alpha 
    - (n-1)\int_0^\eta v^2r^{\alpha-1}\dd r 
    - \int_0^\eta fv_rr^\alpha\dd r 
  \end{split} 
  \end{equation}
  and similarly 
  \begin{equation}
    \label{eq: wryr2n-2}
  \begin{split}
    & - \int_0^\eta w_ryr^\alpha\dd r 
    = - \int_0^\eta (- y_r + v_r - f_r)yr^\alpha\dd r \\
    &=  \int_0^\eta  yy_rr^\alpha\dd r 
    - \int_0^\eta (v_rr^{n-1})_r v_rr^{n-1}\dd r
    + \int_0^\eta  f_ryr^\alpha\dd r \\
    &= \frac{1}{2}y^2(\eta)\eta^\alpha 
    - (n-1)\int_0^\eta y^2r^{\alpha-1}\dd r 
    - \frac12v_r^2 (\eta)\eta^\alpha + \int_0^\eta f_ryr^\alpha\dd r.
  \end{split} 
  \end{equation}
    Collecting \eqref{eq: test by rn-1vr}--\eqref{eq: wryr2n-2} and  
  multiplying by $\eta^{1-n}$, 
  we get 
  \begin{equation}
    \label{eq: 1/2y2etan-1}
    \begin{split}
      &\frac{1}{2}y^2(\eta)\eta^{n-1}
      + w_r(\eta)v_r(\eta)\eta^{n-1}
      - \frac{n-1}{\eta^{n-1}}\int_0^\eta y^2r^{\alpha-1}\dd r \\
      =& \frac12 v^2(\eta)\eta^{n-1} 
      - \frac{1}{\eta^{n-1}}\int_0^\eta uv_rr^\alpha\dd r
      - \frac{1}{\eta^{n-1}}\int_0^\eta fv_rr^\alpha\dd r\\
      &\quad - \frac{1}{\eta^{n-1}}\int_0^\eta f_ryr^\alpha\dd r
      - \frac{n-1}{\eta^{n-1}}\int_0^\eta v^2r^{\alpha-1}\dd r \\
      \leq &\frac12 v^2(\eta)\eta^{n-1} 
      - \frac{1}{\eta^{n-1}}\int_0^\eta uv_rr^\alpha\dd r
      + \omega_n^{-1}\|f\|_{L^2(B_\eta)}\|\nabla v\|_{L^2(B_\eta)}\\
      &\quad + \omega_n^{-1}\|\nabla f\|_{L^2(B_\eta)}\|\Delta v\|_{L^2(B_\eta)} 
    \end{split}
  \end{equation}
  for all $\eta\in(0,R)$,
  where the last inequality follows from dropping the last term and H\"older inequality.
  Integrating \eqref{eq: 1/2y2etan-1} over $(0,\rho)\subset(0,R)$, we have 
  \begin{equation}
    \label{eq: 1/2inty2etan-1}
    \begin{split}
      &\frac12\int_0^\rho y^2\eta^{n-1}\dd\eta 
      + \int_0^\rho w_rv_r\eta^{n-1}\dd\eta 
      - \int_0^\rho\frac{n-1}{\eta^{n-1}}\int_0^\eta y^2r^{\alpha-1}\dd r\dd\eta\\
      \leq &\frac12\int_0^\rho v^2\eta^{n-1}\dd\eta 
      - \int_0^\rho\frac{1}{\eta^{n-1}}\int_0^\eta uv_rr^\alpha\dd r\dd\eta\\
      &\quad + \omega_n^{-1}\rho\|\nabla f\|_{L^2(B_\rho)}\|\Delta v\|_{L^2(B_\rho)} 
      + \omega_n^{-1}\rho \|f\|_{L^2(B_\rho)}\|\nabla v\|_{L^2(B_\rho)}.
    \end{split}
  \end{equation}
  We calculate the second term on the left of \eqref{eq: 1/2inty2etan-1} by integration by parts, 
  the identity $w = -y +v -f$ from \eqref{sym: f} and H\"older inequality 
  \begin{equation} 
    \label{eq: wrvretan-1}
  \begin{split} 
    &\int_0^\rho w_rv_r\eta^{n-1}\dd\eta 
    = wv_r\rho^{n-1} - \int_0^\rho wy\eta^{n-1}\dd\eta \\
    &= wv_r\rho^{n-1}
    + \int_0^\rho y^2\eta^{n-1}\dd\eta 
    - \int_0^\rho vy\eta^{n-1}\dd\eta 
    + \int_0^\rho fy \eta^{n-1}\dd\eta\\
    &= wv_r\rho^{n-1} 
    - vv_r\rho^{n-1}
    + \int_0^\rho y^2\eta^{n-1}\dd\eta 
    + \int_0^\rho v_r^2\eta^{n-1}\dd\eta 
    + \int_0^\rho fy \eta^{n-1}\dd\eta\\
    &\geq \int_0^\rho y^2\eta^{n-1}\dd\eta 
    + \omega_n^{-1} \|\nabla v\|_{L^2(B_\rho)}^2
    - \omega_n^{-1} \|\Delta v\|_{L^2(B_\rho)}\|f\|_{L^2(B_\rho)} 
    + (w-v)v_r\rho^{n-1}
  \end{split} 
\end{equation}
for all $\rho\in(0,R)$.
  By Fubini theorem, we compute the last term on the left of \eqref{eq: 1/2inty2etan-1}
  \begin{equation}
    \begin{split}
      &- \int_0^\rho\frac{n-1}{\eta^{n-1}}\int_0^\eta y^2r^{\alpha-1}\dd r\dd\eta 
      = - (n-1)\int_0^\rho\int_r^\rho\frac{\dd\eta}{\eta^{n-1}} y^2r^{\alpha-1}\dd r\\
      = &- \frac{n-1}{n-2}\int_0^\rho y^2r^{n-1}\left(1 - \frac{r^{n-2}}{\rho^{n-2}}\right)\dd r
      \geq - \frac{n-1}{n-2}\int_0^\rho y^2r^{n-1}\dd r.
    \end{split}
  \end{equation}
  Thanks to $uv_r = u_r - g\sqrt{u}$ given by \eqref{sym: g}, 
  the second term on the right of \eqref{eq: 1/2inty2etan-1} can be controlled accordingly 
  \begin{equation}
    \label{eq: I_3}
  \begin{split} 
    &- \int_0^\rho\frac{1}{\eta^{n-1}}\int_0^\eta uv_rr^\alpha\dd r\dd\eta
    = - \int_0^\rho\frac{1}{\eta^{n-1}}\int_0^\eta (u_r-g\sqrt{u})r^\alpha\dd r \dd\eta\\
    \leq &- \int_0^\rho u\eta^{n-1}\dd\eta 
    + \alpha\int_0^\rho\frac{1}{\eta^{n-1}}\int_0^\eta ur^{\alpha-1}\dd r\dd\eta 
    + \int_0^\rho \int_0^\eta |g\sqrt u|r^{n-1}\dd r\dd\eta\\ 
    \leq &- \int_0^\rho u\eta^{n-1}\dd\eta 
    + \frac{\alpha}{n-2}\int_0^\rho u\eta^{n-1}\left(1-\frac{\eta^{n-2}}{\rho^{n-2}}\right)\dd\eta 
    + \frac{\sqrt m\rho}{\omega_n}\|g\|_{L^2(B_\rho)}\\
    \leq &\  \frac{nm}{\omega_n(n-2)} + \frac{\sqrt{m}\rho}{\omega_n}\|g\|_{L^2(B_\rho)}.
  \end{split} 
  \end{equation} 
  Substituting \eqref{eq: wrvretan-1}--\eqref{eq: I_3} into \eqref{eq: 1/2inty2etan-1} we get  
  \begin{equation} 
    \label{eq: 3/2-n-1/n-2}
  \begin{split}
    &\left(\frac{3}{2} - \frac{n-1}{n-2}\right) \|\Delta v\|_{L^2(B_\rho)}^2 
    + \|\nabla v\|_{L^2(B_\rho)}^2 \\
    &\leq \rho\left(\|\nabla f\|_{L^2(B_\rho)}\|\Delta v\|_{L^2(B_\rho)}
    + \|\nabla v\|_{L^2(B_\rho)}\|f\|_{L^2(B_\rho)}\right)
    + \|\Delta v\|_{L^2(B_\rho)}\|f\|_{L^2(B_\rho)} \\
    & \quad 2m + \sqrt{m}\rho\|g\|_{L^2(B_\rho)} 
    + \|v\|_{L^2(B_\rho)}^2
    + \omega_n (v-w)v_r\rho^{n-1}.
  \end{split} 
  \end{equation}
  The pointwise restrictions \eqref{eq: pointwise constraint for v} imply 
  \begin{displaymath}
    \omega_n (v-w)v_r\rho^{n-1}
    \leq \omega_n(B\rho^{2-\kappa} + M\rho^{2-n})B\rho^{1-\kappa}\rho^{n-1} 
    = \omega_nB^2\rho^{n+2-2\kappa} + \omega_nBM\rho^{2-\kappa}.
  \end{displaymath}
  Young inequality gives 
  \begin{displaymath}
    \rho\|\nabla f\|_{L^2(B_\rho)}\|\Delta v\|_{L^2(B_\rho)}
    \leq \frac{1}{48} \|\Delta v\|_{L^2(B_\rho)}^2 
    + 12\rho^2\|\nabla f\|_{L^2(B_\rho)}^2,
  \end{displaymath}
  and 
  \begin{displaymath}
    \rho\|\nabla v\|_{L^2(B_\rho)}\|f\|_{L^2(B_\rho)} 
    \leq \frac{1}{4}\|\nabla v\|_{L^2(B_\rho)}^2 
    + \rho^2\|f\|_{L^2(B_\rho)}^2
  \end{displaymath}
  as well as 
  \begin{displaymath}
    \|\Delta v\|_{L^2(B_\rho)}\|f\|_{L^2(B_\rho)} 
    \leq \frac{1}{48}\|\Delta v\|_{L^2(B_\rho)} + 12\|f\|_{L^2(B_\rho)}^2.
  \end{displaymath}
  Noting the fact 
  \begin{displaymath}
    \frac{3}{2} - \frac{1}{24} - \frac{n-1}{n-2} \geq \frac18,
  \end{displaymath}
  thanks to $n\geq5$,
  we plug five estimates above into \eqref{eq: 3/2-n-1/n-2},
  which becomes \eqref{eq: 1/8deltavL2} by putting 
  $C_0 := \max\{\omega_n, 12 + R^2\}$.
\end{proof}

We shall eliminate the quadratic zeroth-order terms $\|f\|_{L^2(\Omega)}^2$ and $\|v\|_{L^2(\Omega)}^2$ left in Lemma~\ref{le: uvL1smallball} and  Lemma~\ref{le: deltavL2} without treatment. 

\begin{lemma}
  \label{le: 1/24mathcalF}
  Let $\Omega = B_R\subset\mathbb R^n$ be a ball of radius $R>0$ with some  $n\geq5$. 
  Then there exists a constant $C_4 > 0$ such that 
  \begin{equation*}
    \begin{aligned}
      -\frac{1}{24}\mathcal{F} 
      &\leq 12\rho^2\|\nabla f\|_{L^2(\Omega)}^2 
      + C_4B^{4/(n+2)}\|\nabla f\|_{L^2(\Omega)}^{2n/(n+2)}
      + \sqrt m\rho\|g\|_{L^2(\Omega)} \\
      &\quad + C_4(R^{2\kappa - n} + 1)(B^2 + m^2 + M^2 + 1)\rho^{n-2\kappa}
    \end{aligned}
   \end{equation*}
   holds for all $(u,v)\in \mathcal{S}(m,M,B,\kappa)$ and $\rho\in(0,R)$, 
   where $\mathcal{F}$ is given in \eqref{sym: mathacl F D}.
\end{lemma}

\begin{proof}
  Lemma~\ref{le: firstobservation} and \ref{le: uvL1smallball} entail
  there exists $C_1 > 0$ such that   
  \begin{align*}
    -\mathcal{F} 
    &\leq 3 \|\Delta v\|_{L^2(B_\rho)}^2  
    + 3\|v\|_{L^2(B_\rho)}^2
    + \|f\|_{L^2(B_\rho)}^2
    + C_1(Bm + BM)\rho^{2-\kappa}
    + C_1
  \end{align*}
  holds for all $\rho\in(0,R)$.
  Furthermore, Lemma~\ref{le: deltavL2} warrants the existence of $C_3>0$ satisfying 
  \begin{equation} 
    \label{eq: mathcalFleq}
  \begin{split}
    &-\frac{1}{24}\mathcal{F} 
    \leq 12\rho^2\|\nabla f\|_{L^2(B_\rho)}^2 
    + C_3\|f\|_{L^2(B_\rho)}^2
    + \sqrt m\rho\|g\|_{L^2(B_\rho)} \\
    & + \frac{9}{8}\|v\|_{L^2(B_\rho)}^2
    - \frac34\|\nabla v\|_{L^2(B_\rho)}^2 
    + C_3B^2\rho^{n-2\kappa}
    + C_3(Bm + BM)\rho^{2-\kappa} + 2m + C_1
  \end{split}
  \end{equation} 
  for all $\rho\in(0,R)$.
  Gagliardo-Nirenberg inequality \cite{Nirenberg1959} provides a constant $C_{gn} > 0$ such that 
  \begin{displaymath}
    \|\psi\|_{L^2(\Omega)}^2
    \leq C_{gn} \|\nabla \psi\|_{L^2(\Omega)}^{2n/(n+2)}\|\psi\|_{L^1(\Omega)}^{4/(n+2)} 
    + C_{gn}\|\psi\|_{L^1(\Omega)}^2
  \end{displaymath}
  holds for all $\psi\in W^{1,2}(\Omega)$.  
  Recalling \eqref{eq: constraint of mass} we thus have  
  \begin{equation}
    \label{eq: f gn}
    \|f\|_{L^2(\Omega)}^2 
    \leq C_{gn}B^{4/(n+2)} \|\nabla f\|_{L^2(\Omega)}^{2n/(n+2)}
    + C_{gn}B^2,
  \end{equation}
  as well as 
  \begin{equation}
    \label{eq: v gn}
    \begin{split} 
    \|v\|_{L^2(B_\rho)}^2 
    &\leq C_{gn}M^{4/(n+2)} \|\nabla v\|_{L^2(\Omega)}^{2n/(n+2)} 
    + C_{gn} M^2\\
    &\leq \frac{1}{2}\|\nabla v\|_{L^2(\Omega)}^2
    + 2^nC_{gn}^{(n+2)/2}M^2
    + C_{gn} M^2
    \end{split}
  \end{equation}
  for all $\rho\in(0,R)$
  by Young inequality.
  Invoking \eqref{eq: pointwise constraint for v}, we estimate 
  \begin{equation}
    \label{eq: nablavouterregion}
    \begin{split} 
      \|\nabla v\|^2_{L^2(\Omega)}
      &= \|\nabla v\|_{L^2(B_\rho)}^2 
      + \|\nabla v\|^2_{L^2(\Omega\setminus B_\rho)}\\
      &\leq \|\nabla v\|_{L^2(B_\rho)}^2  
      +B^2\||x|^{1-\kappa}\|_{L^2(\Omega\setminus B_\rho)}^2\\
      &\leq \|\nabla v\|_{L^2(B_\rho)}^2 
      + \omega_nB^2R^2\rho^{n-2\kappa}
    \end{split}
  \end{equation}
  for all $\rho\in(0,R)$.
 Inserting \eqref{eq: f gn}--\eqref{eq: nablavouterregion} into \eqref{eq: mathcalFleq} yields 
 \begin{equation*}
  \begin{split}
    -\frac{1}{24}\mathcal{F} 
    &\leq 12\rho^2\|\nabla f\|_{L^2(\Omega)}^2 
    + C_3C_{gn}B^{4/(n+2)}\|\nabla f\|_{L^2(\Omega)}^{2n/(n+2)}
    + \sqrt m\rho\|g\|_{L^2(\Omega)} + C(\rho) 
  \end{split}
 \end{equation*}
 for all $\rho\in(0,R)$, 
 where 
 \begin{align*}
 C(\rho) &:= C_3C_{gn}B^2
 + 2^{n+1}C_{gn}^{(n+2)/2}M^2
 + 2C_{gn} M^2
 + \omega_nB^2R^2\rho^{n-2\kappa}\\
 & \quad + C_3B^2\rho^{n-2\kappa}
 + C_3(Bm + BM)\rho^{2-\kappa} + 2m + C_1 \\
 &\:\leq (C_3C_{gn} + 2^{n+1}C_{gn}^{(n+2)/2} + 2C_{gn} + 1)(B^2 + M^2 + m^2 + 1)\\
 &\quad + (\omega_nR^2+C_3 + C_3R^{\kappa-n+2} + C_1R^{2\kappa-n})(B^2+M^2+m^2+1)\rho^{n-2\kappa}.
 \end{align*}
 Taking
 \begin{displaymath}
\begin{split}
  C_4 &: = C_3C_{gn} 
    + 2^{n+1}C_{gn}^{(n+2)/2}
    + 2C_{gn}
    + 1 
    + \omega_n + 2C_3 + C_1,
\end{split}
 \end{displaymath}
 we finish the proof.
\end{proof}

Here we shall choose $\rho$ appropriately, following the strategy of Winkler~\cite{Winkler2013}, 
to establish a lower bound of the energy $\mathcal{F}$ by a sublinear power of the dissipation rate $\mathcal{D}$.

\begin{lemma}
  \label{le: mathcalFDtheta}
  Let $\Omega = B_R\subset\mathbb R^n$ be a ball of radius $R>0$ with some  $n\geq5$.  
  Then there exist $\theta\in(0,1)$ and $C(m, M, B, \kappa) > 0$ such that 
  \begin{equation}
    \label{eq: mathcalFDtheta}
    -\mathcal{F} \leq C(m, M, B, \kappa) (\mathcal{D}^\theta + 1)
  \end{equation}
  holds for all $(u,v)\in \mathcal{S}(m, M, B, \kappa)$, 
  where $\mathcal{F}$ and $\mathcal{D}$ are specified as in \eqref{sym: mathacl F D}.
\end{lemma}

 \begin{proof} 
 Put  
 \begin{displaymath}
  \rho :=
  \min\left\{\frac{R}{2}, \|\nabla f\|_{L^2(\Omega)}^{-\frac{n}{(n+2)(2\kappa-n)}}\right\}, 
 \end{displaymath}
 and let  
 \begin{displaymath}
  \theta := 
  \begin{cases} 
    \frac{n}{n+2} > 1- \frac{n}{(n+2)(2\kappa-n)} > 0, & \text{if } \kappa\in\big(n-2, \frac{3n}4\big),\\
    1- \frac{n}{(n+2)(2\kappa-n)} \geq \frac{n}{n+2}, & \text{if } \kappa\in \big[\frac{3n}4,\infty\big).
  \end{cases}
 \end{displaymath}
 We may infer the existence of $ C(m,M,B,\kappa) > 0$ from Lemma~\ref{le: 1/24mathcalF} such that 
 \begin{equation*}
  \begin{aligned}
    -\mathcal{F} 
    &\leq C(m,M,B,\kappa)(\rho^2\|\nabla f\|_{L^2(\Omega)}^2 
    + \|\nabla f\|_{L^2(\Omega)}^{2n/(n+2)}
    + \|g\|_{L^2(\Omega)}
    + \rho^{n-2\kappa})
  \end{aligned}
 \end{equation*}
 holds for all $(u,v)\in \mathcal{S}(m,M,B,\kappa)$.
 If $\rho < R/2$, then 
 \begin{equation}
  \label{eq: rho=R/2}
  \begin{split}
    -\mathcal{F} 
    &\leq C(m,M,B,\kappa)(\rho^2\|\nabla f\|_{L^2(\Omega)}^2 
    + \|\nabla f\|_{L^2(\Omega)}^{2n/(n+2)}
    + \|g\|_{L^2(\Omega)}
    + \|\nabla f\|_{L^2(\Omega)}^{n/(n+2)})\\
    &\leq C(m,M,B,\kappa)\left(3\|\nabla f\|_{L^2(\Omega)}^{2\theta} 
    + \|g\|_{L^2(\Omega)}^{2\theta} + 4\right)\\
    &\leq C(m,M,B,\kappa) (\mathcal{D}^\theta + 1),
  \end{split}
 \end{equation}
 by Young inequality and 
 monotonicity of the mapping $[0,\infty)\ni\eta\mapsto\eta^\theta\in[0,\infty)$.
 On the other hand, if $\rho = R/2$, then we can estimate analogously 
 \begin{equation}
  \label{eq: rho<R/2}
  \begin{split}
    - \mathcal{F} 
    &\leq C(m,M,B,\kappa)\left(\rho^2\|\nabla f\|_{L^2(\Omega)}^2 
    +  \|\nabla f\|_{L^2(\Omega)}^{2n/(n+2)}
    + \|g\|_{L^2(\Omega)}
    + \frac{2^{2\kappa-n}}{R^{2\kappa-n}}\right)\\
    &\leq C(m,M,B,\kappa)\left(2\|\nabla f\|_{L^2(\Omega)}^{2\theta} 
    + \|g\|_{L^2(\Omega)}^{2\theta} + 3
    + \frac{2^{2\kappa-n}}{R^{2\kappa-n}}\right)\\
    &\leq C(m,M,B,\kappa) (\mathcal{D}^\theta + 1).
  \end{split}
 \end{equation}
 Hence, \eqref{eq: mathcalFDtheta} follows from \eqref{eq: rho=R/2} and \eqref{eq: rho<R/2}.
\end{proof}

\section{Finite-time blowup enforced by low initial energy}
\label{sec: finite-time blowup} 

This section is devoted to the proof of Theorem~\ref{thm: low energy enforced finite-time blowup}. 
For initial data $(u_0, v_0)$ with low energy $\mathcal{F}(u_0,v_0)$, we may deduce a superlinear differential inequality for the mapping $t\mapsto - \mathcal{F}(u(\cdot, t),v(\cdot, t))$ from \eqref{eq: mathcalFDtheta},
which implies that $(u,v)$ cannot exist globally.

\begin{lemma}
  \label{le: low energy enforced finite-time blowup}
  Assume $\Omega = B_R \subset\mathbb R^n$ for some $n\geq5$ and $R>0$.
  Let $m > 0$, $A > 0$ and $\kappa > n-2$ be given.
  Then there exist positive constants $K(m,A, \kappa)$ and $T(m,A, \kappa)$ such that 
  given any $(u_0,v_0)$ from the set 
  \begin{equation}
    \label{sym: tilde B}
    \begin{split}
       \widetilde{\mathcal{B}}(m,A,\kappa) &:= \{(u_0,v_0)\in  C^0(\overline{\Omega}) \times C^2(\overline{\Omega})\mid 
        u_0 \text{ and } v_0 \text{ comply with } \eqref{h: initial data} 
       \text{ and }\\ 
       &\eqref{h: radial initial data} 
        \text{ such that } \|u_0\|_{L^1(\Omega)} = m, 
        \|v_0\|_{W^{2,2}(\Omega)} < A\\ 
        &\text{and } \mathcal{F}(u_0,v_0) < - K(m,A,\kappa)\},
    \end{split}
  \end{equation}
  for the corresponding classical solution with maximal existence time $T_{\max}(u_0,v_0)$ given by Proposition~\ref{prop: local existence and uniqueness},
  one has $T_{\max}(u_0,v_0) \leq T(m,A, \kappa)$.
\end{lemma}

\begin{proof}
  Let us fix $c_1>0$ such that 
  \begin{displaymath}
    \|\psi\|_{L^1(\Omega)} \leq c_1 \|\psi\|_{W^{2,2}(\Omega)} 
    \quad\text{for all } \psi\in W^{2,2}(\Omega).
  \end{displaymath}
  Lemma~\ref{le: basic estimates} and Proposition~\ref{prop: pointwise estimate} give $c_2 > 0$, $c_3 > 0$ and $c_4 = c_4(\kappa) > 0$ such that 
  whenever $(u_0,v_0)$ complies with \eqref{h: initial data}, the corresponding solution $(u,v,w)$ satisfies $L^1$ estimate 
  \begin{equation}
    \label{eq: vtL1c2}
    \|v_t\|_{L^1(\Omega)} \leq c_2(m + \|v_0\|_{W^{2,2}(\Omega)})
    \quad\text{for all } t \in (0, T_{\max}(u_0,v_0)),
  \end{equation}
    as well as pointwise estimates 
  \begin{equation}
    \||x|^{n-1}w\|_{W^{1,\infty}(\Omega)} 
    \leq c_3m
    \quad\text{for all } t \in (0, T_{\max}(u_0,v_0)),
  \end{equation}
  and  
  \begin{equation}
    \label{eq: xkappavw2inffty}
    \||x|^{\kappa-1} v\|_{W^{1,\infty}(\Omega)} 
    \leq c_4(m + \|v_0\|_{W^{2,2}(\Omega)})
  \end{equation}
  for all $t\in(0,T_{\max}(u_0,v_0))$.
  Now writing 
  \begin{align*} 
  B&:= \max\{c_1A, m, c_3m\} \quad\text{and} 
  \quad M := \max\{c_2(m + A), c_4(m + A)\},
  \end{align*}
  we invoke Lemma~\ref{le: mathcalFDtheta} to obtain $c_5 = C(m,M,B,\kappa) > 0$ and $\theta\in(0,1)$ such that 
  \begin{equation}
    \label{eq: mathcalFc5D}
    -\mathcal{F}(\tilde u, \tilde v) \leq c_5 (\mathcal{D}^\theta(\tilde u, \tilde v) + 1)
    \quad \text{for all } (\tilde u, \tilde v)\in\mathcal{S}(m,M,B,\kappa).
  \end{equation}
  We will show the desired conclusion holds if we define 
  \begin{equation}
    \label{eq: Kmackappa}
    K(m,A, \kappa) := c_5
  \end{equation}
  and 
  \begin{equation*}
    T(m,A, \kappa) := \frac{\theta c_5^{\frac{1}{\theta}}}{(1-\theta)(-\mathcal F(u_0,v_0) - c_5)^{\frac{1-\theta}{\theta}}}.
  \end{equation*}
  Indeed, we know from Proposition~\ref{prop: local existence and uniqueness} that for any given $(u_0,v_0)\in \widetilde{\mathcal{B}}(m,A, \kappa)$,
  the corresponding solution $(u,v,w)$ is smooth and radially symmetric with $u>0$ and $v>0$ in $\overline{\Omega}\times[0,T_{\max}(u_0,v_0))$.
  We infer from \eqref{eq: vtL1c2}--\eqref{eq: xkappavw2inffty} that 
  \begin{equation*}
    \|v_t\|_{L^1(\Omega)}\leq c_2(m + A) \leq M 
    \quad\text{for all } t \in (0, T_{\max}(u_0,v_0)),
  \end{equation*}
    as well as pointwise estimates 
  \begin{equation*}
    \||x|^{n-1}w\|_{W^{1,\infty}(\Omega)} \leq c_3m \leq B 
    \quad\text{for all } t \in (0, T_{\max}(u_0,v_0)),
  \end{equation*}
  and  
  \begin{equation*}
    \||x|^{\kappa-1} v\|_{W^{1,\infty}(\Omega)} 
    \leq c_4(m + A)
    \leq M 
  \end{equation*}
  for all $t\in(0,T_{\max}(u_0,v_0))$.
  Since moreover 
  $\int_{\Omega} u(x, t) =\int_{\Omega} u_{0} = m$ and 
  \begin{displaymath} 
  \int_\Omega v(x,t) 
    \leq \max\left\{\int_\Omega v_0, \int_\Omega u_0\right\} 
    \leq \max\{c_1A, m\} \leq B
  \end{displaymath}
  for all  $t\in(0,T_{\max}(u_0,v_0))$ by \eqref{eq: mass identity} and \eqref{eq: vmassinequality}, 
  it follows that $(u,v)\in \mathcal{S}(m,M,B,\kappa)$ for all $t\in(0,T_{\max}(u_0,v_0))$ and thus \eqref{eq: mathcalFc5D} may be applied to 
  $(\tilde{u},\tilde{v}) = (u(\cdot, t), v(\cdot,t))$ for all $t\in(0,T_{\max}(u_0,v_0))$.
  From \eqref{eq: energyequation} and \eqref{eq: Kmackappa}, we see 
  that 
  \begin{displaymath}
     \mathcal F(u(\cdot,t),v(\cdot,t)) 
    \leq \mathcal F(u_0,v_0) < - K(m,A, \kappa) = - c_5 
    \quad\text{for all } t \in (0, T_{\max}(u_0,v_0)),
  \end{displaymath}
  and thus we may invert \eqref{eq: mathcalFc5D} to get 
  \begin{equation}
    \label{eq: mathcalFODI}
    - \frac{\dd}{\dd t}\mathcal F = \mathcal D 
    \geq \left(\frac{-\mathcal F - c_5}{c_5}\right)^\frac{1}{\theta}
    \quad\text{for all } t \in (0, T_{\max}(u_0,v_0)).
  \end{equation}
  We shall see $T_{\max}(u_0,v_0) \leq T(m,A, \kappa)$.
  Otherwise a direct comparison argument by solving \eqref{eq: mathcalFODI} gives 
  \begin{displaymath}
    (-\mathcal F(u,v) - c_5)^{\frac{1-\theta}{\theta}} 
    \geq \frac{\theta c_5^{\frac{1}{\theta}}}{(1-\theta)\left(T(m,A, \kappa) - t\right)}
  \end{displaymath}
  for all $t\in(0,T(m,A, \kappa))$. 
  This warrants 
  \begin{displaymath}
  \mathcal F(u,v)\to-\infty \quad\text{as } t\nearrow T(m,A, \kappa),
  \end{displaymath} 
  which is incompatible with Proposition~\ref{prop: local existence and uniqueness}.
  We thus verify $T_{\max}(u_0,v_0) \leq T(m,A, \kappa)$ as desired.
\end{proof}

Now we are in a position to show Theorem~\ref{thm: low energy enforced finite-time blowup}.

\begin{proof}
  [Proof of Theorem~{\upshape\ref{thm: low energy enforced finite-time blowup}}]  
  We only need to fix an arbitrary $\kappa>n-2$ and apply Lemma~\ref{le: low energy enforced finite-time blowup} to find that the conclusion holds if we let $K(m, A):=K(m, A, \kappa)$ and $T(m, A):=T(m,A, \kappa)$ with $K(m, A, \kappa)$ and $T(m,A, \kappa)$ given by Lemma~\ref{le: low energy enforced finite-time blowup}.
\end{proof}

\section{Initial data with large negative energy}
\label{sec: initial data}

This section is devoted to the existence of initial data with low energy. 

\begin{lemma}
  \label{le: a family initial data}
  Let $n\geq5$ and $\Omega \subset \mathbb R^n$ be a bounded and smooth domain. 
  Then for any choice of function pair $(u_0,v_0)$ satisfying \eqref{h: initial data},
  one can find a sequence
  \begin{displaymath}
    \{(u_k, v_k) \in L^\infty(\Omega) \times W^{2,\infty}(\Omega)\}_{k\in\mathbb{N}}
  \end{displaymath}
  of function couples such that 
  $(u_k, v_k)$ satisfies \eqref{h: initial data} for all $k\in\mathbb{N}$,
  that 
  \begin{equation}
    \label{eq: intueta=m}
    \int_\Omega u_k = \int_\Omega u_0 
    \quad \text{for all } k\in\mathbb{N},
  \end{equation}  
  that 
  \begin{equation}
    \label{eq: vetatov0}
    v_k \to v_0 \quad \text{in } W^{2, 2}(\Omega) 
    \text{ as } k\to\infty,
  \end{equation}
  that 
  \begin{equation}
    \label{eq: uetatou0}
    u_k \to u_0 \quad \text{in } L^p(\Omega) 
    \text{ for all } p\in\bigg[1,\frac{2n}{n+4}\bigg)
    \text{ as } k\to\infty,
  \end{equation}
  and that 
  \begin{equation}
    \label{eq: mathcalFuetaveta}
    \mathcal F(u_k, v_k)\to - \infty,\quad \text{as } k \to 0,
  \end{equation}
  where $\mathcal F$ is taken from \eqref{sym: mathcalFD}.
\end{lemma}

\begin{proof}
  Without loss of generality, we may assume that $B_R \subset \Omega$ for some $R>0$.
  Let $\phi\in C_0^\infty(\mathbb R^n)$ be a radially symmetric and nonnegative function with $\int_{\mathbb R^n}\phi = 1$, 
which is compactly supported in $B_1$.
Fixing $\gamma > 0$,
we define for $k\in\mathbb{N}$,
  \begin{displaymath}
    u_k(x) := u_0(x) 
    + \left(\ln\frac1\eta\right)^{2\gamma}\eta^{-\frac{n}{2}-2}\phi\left(\frac{x}{\eta}\right) 
    - \left(\ln\frac1\eta\right)^{2\gamma}\eta^{\frac{n}{2}-2}\cdot\frac{u_0(x)}{\int_\Omega u_0},
    \quad x\in\overline{\Omega},
  \end{displaymath}
  and 
  \begin{displaymath}
    v_k(x) := v_0(x) 
    + \left(\ln\frac1\eta\right)^{-\gamma} 
    \eta^{2-\frac{n}{2}}
    \phi\left(\frac{x}{\eta}\right),\quad x\in\overline{\Omega},
  \end{displaymath}
  where %
  \begin{displaymath}
    \eta = \eta(k) := \frac{\eta_\star}{k+1}, \quad k\in\mathbb{N},
  \end{displaymath}
  and $\eta_\star \in (0,1)\cap(0,R)$ is fixed such that  
  \begin{equation}
    \label{eq: etafornonnegative}
    \eta^{\frac{n}{2}-2}\left(\ln\frac1\eta\right)^{2\gamma} 
    < \int_\Omega u_0 
    \quad \text{for all } k\in\mathbb{N}.
  \end{equation}
A direct computation can verify the mass identity \eqref{eq: intueta=m}.
Since 
\begin{displaymath}
  \|v_k-v_0\|_{W^{2,2}(\Omega)}
  \leq \left(\ln\frac1\eta\right)^{-\gamma} \|\phi\|_{W^{2,2}(B_1)}
\end{displaymath}
are valid for all $k\in\mathbb{N}$,
$\gamma > 0$ ensures \eqref{eq: vetatov0}.
We estimate by triangle inequality,
\begin{displaymath}
  \begin{split}
    \|u_k - u_0\|_{L^p(\Omega)} 
    &\leq \left(\ln\frac1\eta\right)^{2\gamma} \eta^{-2-\frac{n}{2}+\frac{n}{p}}\|\phi\|_{L^p(B_1)}\\
    &\quad + \left(\ln\frac1\eta\right)^{2\gamma}\eta^{\frac{n}{2}-2}\cdot\frac{\|u_0\|_{L^p(\Omega)}}{\int_\Omega u_0}
    \quad \text{for all } k\in\mathbb{N} \text{ and } p\in[1,\infty).
  \end{split}
\end{displaymath}
Then \eqref{eq: uetatou0} follows due to $-2-n/2+n/p>0$ and $n>4$.
To verify \eqref{eq: mathcalFuetaveta},
we evaluate each term from $\mathcal{F}$.
Young inequality entails 
\begin{equation}
  \label{eq: veta2}
  \int_\Omega |v_k|^2 
  \leq 2\|v_0\|_{L^2(\Omega)}^2 
  + 2 \left(\ln\frac1\eta\right)^{-2\gamma}\eta^{4}\|\phi\|_{L^2(B_1)}^2
  \quad \text{for all } k\in\mathbb{N}
\end{equation}
and 
  \begin{equation}
    \label{eq: deltav2}
    \int_\Omega |\Delta v_k|^2 
    \leq 2\|\Delta v_0\|_{L^2(\Omega)}^2 
    + 2\left(\ln\frac1\eta\right)^{-2\gamma}\|\Delta \phi\|_{L^2(B_1)}^2
    \quad \text{for all } k\in\mathbb{N}.
  \end{equation}
Thanks to \eqref{eq: etafornonnegative}, 
$u_k$ is nonnegative for all $k\in\mathbb{N}$,
thus we can drop the nonnegative summands and estimate from below 
  \begin{equation}
  \label{eq: uetaveta}
    \begin{split}
      \int_\Omega u_k v_k 
      &\geq \left(\ln\frac1\eta\right)^{\gamma}\int_\Omega \phi^{2}\left(\frac{x}{\eta}\right)\eta^{-n}\dd x \\
      &= \left(\ln\frac1\eta\right)^{\gamma} \|\phi\|_{L^2(B_1)}^2
      \quad \text{for all } k\in\mathbb{N}.
    \end{split}
  \end{equation}
  Noting that 
  \begin{displaymath}
    \xi^q -\xi \geq (q-1)\xi\ln\xi
    \quad \text{for all } \xi\in(0,\infty) \text{ and } q \in (1,\infty),
  \end{displaymath}
  we compute 
  \begin{equation}
    \label{eq: ln1+ln1/eta}
  \begin{split} 
  \int_\Omega u_k\ln u_k 
  &\leq \int_\Omega \frac{u_k^q-u_k}{q-1}
  \leq \int_\Omega \frac{u_k^q}{q-1}\\ 
  &\leq \frac{2^{q-1}}{q-1}\bigg(\|u_0\|_{L^q(\Omega)}^q 
  + \Big(\ln\frac{1}{\eta}\Big)^{2q\gamma}\eta^{-\frac{nq}{2}-2q + n}\|\phi\|_{L^q(B_1)}^q\bigg)
  \end{split}
\end{equation}
for all $k\in\mathbb{N}$ and $q\in(1,\infty)$.
Collecting \eqref{eq: veta2}--\eqref{eq: ln1+ln1/eta} gives  
\begin{equation*}
  \begin{split} 
  \mathcal{F}(u_k, v_k)
  &= \int_{\Omega}(u_k \log u_k - u_k v_k) 
  + \frac{1}{2} \int_{\Omega}|(-\Delta+1) v_k|^2\\ 
  &\leq - \left(\ln\frac1\eta\right)^{\gamma} \|\phi\|_{L^2(B_1)}^2
  + 2 \left(\ln\frac1\eta\right)^{-2\gamma}(\|\Delta \phi\|_{L^2(B_1)}^2 + \|\phi\|_{L^2(B_1)}^2)\\
  &\quad + \frac{2^{q-1}}{q-1}\bigg(\ln\frac{1}{\eta}\bigg)^{2q\gamma}\eta^{-\frac{nq}{2}-2q + n}\|\phi\|_{L^q(B_1)}^q\\
  &\quad + 2\|\Delta v_0\|_{L^2(\Omega)}^2 
  + 2\|v_0\|_{L^2(\Omega)}^2 
  + \frac{2^{q-1}}{q-1}\|u_0\|_{L^q(\Omega)}^q 
  \end{split}
\end{equation*}
for all $k\in\mathbb{N}$
and 
\begin{displaymath}
  q\in\bigg(1,\frac{2n}{n+4}\bigg).
\end{displaymath}
In view of $\gamma > 0$, 
we have $\mathcal{F}(u_k, v_k)\to-\infty$ as $k\to\infty$ and complete the proof.
\end{proof}

Now we are in a position to show Theorem~\ref{thm: dense initial data for blowup}.

\begin{proof}[Proof of Theorem~{\upshape\ref{thm: dense initial data for blowup}}]
  Let initial datum $(u_0,v_0)$ be specified as in Theorem~\ref{thm: dense initial data for blowup}.
Then the sequence $\{(u_k, v_k)\}_{k\in\mathbb{N}}$ of positive function pairs constructed in Lemma~\ref{le: a family initial data}
clearly satisfies the desired properties according to \eqref{eq: intueta=m}--\eqref{eq: mathcalFuetaveta}. 
\end{proof}

\backmatter

\bmhead*{Acknowledgements}
The authors thank the anonymous referees for their helpful
comments and suggestions, which greatly improve the presentation of our paper.

\bmhead*{Funding}
The first author has been supported by ``the Fundamental Research Funds for the Central Universities'' (No.~B250201215).
The second author has been supported in part by National Natural Science Foundation of China (No. 12271092, No. 11671079) 
and the Jiangsu Provincial Scientific Research Center of Applied Mathematics under Grant No. BK20233002.

\bmhead*{Conflict of interest} 

The authors have not disclosed any conflict of interest.

\bibliography{ks-isp-fb-amo}%

\end{document}